\documentclass[11pt, twoside, leqno]{article}
\usepackage{amsmath,amsthm}
\usepackage{amssymb,latexsym}
\usepackage{enumerate}
\usepackage[all]{xy}
\usepackage{amscd,multicol}
\usepackage{hyperref}
\usepackage{graphicx,color}
\usepackage{mathtools}
\usepackage{hhline}

\newtheorem{theorem}{Theorem}
\newtheorem{prop}[theorem]{Proposition}
\newtheorem{lemma}[theorem]{Lemma}

\newtheorem{claim}{Claim}

\newtheorem*{cor}{Corollary}

\theoremstyle{definition}
\newtheorem*{definition}{Definition}

\newtheorem{remark}{Remark}

\numberwithin{equation}{section}

\newcommand{\Z}{{\mathbb Z}}
\newcommand{\C}{{\mathbb C}}
\newcommand{\Q}{{\mathbb Q}}
\newcommand{\R}{{\mathbb R}}
\newcommand{\HH}{{\mathbb H}}
\def\wt{\widetilde }

\newcommand{\rk}{\mathop{\mathrm{rank}}\nolimits}

\newcommand{\PrimQ}{{\text{\rm Cob}^{Sp}\text{\rm Prim}}}
\newcommand{\PrimSO}{{\text{\rm Cob}^{SO}\text{\rm Prim}}}
\newcommand{\Prim}[1]{{\text{\rm CobPrim}\Sigma^{1_{#1}}}}
\newcommand{\CobQ}{{\text{\rm Cob}^{Sp}}}
\newcommand{\CobSO}{{\text{\rm Cob}^{SO}}}
\newcommand{\Cob}[1]{{\text{\rm Cob}\Sigma^{1_{#1}}}}
\newcommand{\Sig}[1]{{\Sigma^{1_{#1}}}}
\newcommand{\up}{{\uparrow}}
\allowdisplaybreaks

\begin{document}

\title{Classifying spaces for projections of immersions with controlled singularities}

\renewcommand{\thefootnote}{}

\author{Andr\'as Sz\H{u}cs, Tam\'as Terpai}

\date{}

\maketitle

\footnotetext{Acknowledgement: The authors were supported by the National Research, Development and Innovation Office NKFIH (OTKA) Grant NK 112735. TT was also supported by the National Research, Development and Innovation Office NKFIH (OTKA) Grant K 120697.}


\renewcommand{\thefootnote}{\arabic{footnote}}
\setcounter{footnote}{0}

\begin{abstract}
We give an explicit simple construction for classifying spaces of maps obtained as hyperplane projections of immersions. We prove structure theorems for these classifying spaces.
\end{abstract}

\section{Introduction}

\begin{definition}
Let $M^n$ and $P^{n+k}$ be smooth manifolds and $f:M^n \to P^{n+k}$ a smooth ($\mathcal C^\infty$-) map. $f$ is called a \emph{corank $1$ map} if $\rk df_x \geq n-1$ for all $x \in M^n$. A stable corank $1$ map is called a \emph{Morin map}.
\end{definition}

\begin{definition}
Given a Morin map $f$ we say that $x\in M^n$ is a \emph{$\Sigma^{1_r,0}$-point} if there exists a regular curve $\gamma: (\R,0) \to (M,x)$ going through $x$ that has $\frac{\partial^r(f\circ \gamma)}{\partial t^r}(0)=0$, but no such regular curve satisfies $\frac{\partial^{r+1}(f\circ \gamma)}{\partial t^{r+1}}(0)=0$.
\end{definition}
Morin \cite{Morin} showed that for a fix $r$ all $\Sigma^{1_r,0}$ germs are left-right equivalent ($\mathcal A$-equivalent) and that for $r \neq s$ the $\Sigma^{1_r,0}$ germs are not equivalent to $\Sigma^{1_s,0}$ germs.

\begin{definition}
Given a Morin map $f: M^n \to P^{n+k}$ we denote by $\Sigma(f)$ the set of its singular points and we denote by $\Sigma^{1_r,0}(f)$ the set of its $\Sigma^{1_r,0}$-points.
\end{definition}

\begin{definition}
A Morin map is called a \emph{$\Sig{r}$-map} if it has no $\Sigma^{1_s,0}$-points with $s>r$.
\end{definition}

\begin{definition}
A corank $1$ map $f:M^n \to P^{n+k}$ equipped with a trivialization of its kernel line bundle is called a \emph{prim} map.
\end{definition}

Note that a prim map is the composition of an immersion $g:M^n \looparrowright P^{n+k} \times \R^1$ with the standard projection $pr:P^{n+k} \times \R^1 \to P^{n+k}$.\footnote{The word \emph{prim} is an abbreviation of ``\emph{pr}ojection of \emph{im}mersion''.}


We denote by $\Prim r(P)$ the cobordism group of prim $\Sig{r}$-maps in a fixed target manifold $P$, and we denote by $\Cob r(P)$ the cobordism group of all $\Sig{r}$-maps in a fixed target manifold $P$. For the standard definitions of these groups see \cite{GT}. Analogous groups can be defined for the case of cooriented maps or maps with a quaternionic normal structure; we denote them by $\PrimSO\Sig r(P)$, $\CobSO\Sig r(P)$ and $\PrimQ\Sig r (P)$, $\CobQ\Sig r(P)$, respectively.

Whenever there is a (prim) $\Sig{r}$-map $f:M\to P$ and a smooth map $g:P' \to P$ a standard pullback diagram arises:
$$
\xymatrix{
M' \ar@{-->}[r]^{g^*f} \ar@{-->}[d] & P' \ar[d]^{g} \\
M \ar[r]_f & P
}
$$
If the map $g$ is transverse to all the submanifolds $f\left(\Sigma^{1_s,0}(f)\right)$ for $s=0,\dots,r$, then the map $g^*f$ is a (prim) $\Sig{r}$-map as well. If the map $g$ is not transverse to the submanifolds $\Sigma^{1_s}(f)$, one can still choose an approximating map $\tilde g$ close to $g$ that is transverse and obtain a (prim) $\Sigma^{1_r}$-map $\tilde g^*f$; this map is not unique, but any two choices of such approximating maps --- $\tilde g_1$ and $\tilde g_2$, say --- can be deformed into one another by a homotopy $G$ that is itself transverse to the submanifolds $\Sigma^{1_s}(f)$, hence $\tilde g_1^*f$ and $\tilde g_2^*f$ are connected by the cobordism $G^*f$ and represent the same element in the (prim) $\Sigma^{1_r}$-cobordism group of $P$. It is easy to check that sending $f$ to the pullback map $f\mapsto \tilde g^*f$ induced by a suitable approximation of $g$ on the cobordism group of $P$  we obtain a contravariant functor from the category of smooth manifolds and smooth maps to the category of groups and homomorphisms. A similar functor can be defined using $\Prim r(\cdot)$ instead of $\Cob r(\cdot)$.

There exist (homotopically unique) spaces $X_r$ and $\overline{X}_r$ that represent the functors
\[
\aligned
P &\longrightarrow \Cob r(P) \ \text{ and}\\
P &\longrightarrow \Prim r(P)
\endaligned
\]
in the sense of Brown representability theorem\footnote{In order to apply Brown's theorem directly, one has to extend these functors to arbitrary simplicial complexes (not only manifolds). This is done in \cite{GT}.} (see \cite{Switzer}), in particular
\[
\aligned
\Cob r(P) &= [P,  X_r]\ \text{ and}\\
\Prim r(P) &= [P, \overline X_r]
\endaligned
\]
for any compact manifold $P$.
We call the spaces $X_r$ and $\overline X_r$ the \emph{classifying spaces} for $\Sigma^{1_r}$-maps and prim $\Sigma^{1_r}$-maps, respectively. This type of classifying spaces in a more general setup has been explicitly constructed and investigated earlier, see \cite{Sbornik}, \cite{LNM}; in \cite{RSz} and \cite{GT} two significantly different explicit descriptions of the classifying space for a more general class of singular maps are given, and in \cite{gluing} a homotopy theoretic connection between those constructions is established. Again, analogues for oriented maps and quaternionic maps can be defined, we denote them by $X^{\rm SO}_r$, $\overline X^{\rm SO}_r$, $X^{\rm Sp}_r$ and $\overline X^{\rm Sp}_r$; in what follows, we will omit the distinguishing upper indices when the argument works for each case.

In the present paper we give a simple construction for the spaces $\overline X_r$ and prove structure theorems for them. As a byproduct we get an explicit description of some elements of stable homotopy groups of spheres via local forms of Morin maps.

\section{Construction of $\overline{X_r}$, the classifying space of co\-bor\-disms of prim $\Sigma^{1_r}$-maps}

{\bf Notation:} Let $\gamma^{\rm O}_{k+1} \to BO(k+1)$ denote the universal $(k+1)$-dimensional vector bundle, let
$\gamma^{\rm SO}_{k+1} \to BSO(k+1)$ denote the universal $(k+1)$-dimensio\-nal oriented vector bundle and denote by
$\gamma^{\rm Sp}_{k+1} \to BSp(k+1)$ the universal $4(k+1)$-dimensional quaternionic vector bundle. We denote by $\gamma_{k+1}$ one of these bundles, with the implication that our arguments apply to each case. Let $S=S\left((r+1)\gamma_{k+1}\right)$ be the sphere bundle of $(r+1)\gamma_{k+1} = \overbrace{\gamma_{k+1} \oplus \cdots \oplus \gamma_{k+1}}^{r+1}$, with $pr_S:S \to B(k+1)$ denoting the projection on $B(k+1)$, which is either $BO(k+1)$, $BSO(k+1)$, or $BSp(k+1)$. Define the bundle $\zeta_S$ to be the pullback $pr_S^*\gamma^{k+1}$. The Thom space of $\zeta_S$ will be denoted by $T\zeta_S$, and $\Omega\Gamma T\zeta_S$ is the space $\Omega^{\infty+1}S^\infty T\zeta_S = \underset{q\to\infty}{\lim} \Omega^{q+1}S^q T\zeta_S$.

\begin{theorem}\label{thm:classic}
$\overline{X_r}=\Omega \Gamma T\zeta_S$.
\end{theorem}

From here onward, the symbol $\cong_\Q$ stands for rational homotopy equivalence.

\begin{theorem}\label{thm:fibration}
\begin{enumerate}[a)]
\item There is a fibration $\overline{X_r} \xrightarrow{\Omega^2\Gamma T\left((r+2)\gamma_{k+1}\right)} \Omega\Gamma T\gamma_{k+1}$.
\item If $k$ is odd, then
$$\Omega\Gamma T\gamma^{\rm SO}_{k+1} \cong_\Q \overline X^{\rm SO}_r \times \Omega\Gamma T\left( (r+2)\gamma^{\rm SO}_{k+1}\right)$$
\item If $k$ is even, then
$$\overline X^{\rm SO}_r \cong_\Q \Omega^2\Gamma T\left((r+2)\gamma^{\rm SO}_{k+1}\right) \times \Omega\Gamma T\gamma^{\rm SO}_{k+1}$$
\end{enumerate}
\end{theorem}

The proofs are postponed to Section \ref{section:proofs}.

An important application of Theorem \ref{thm:fibration} is that it allows the calculation of the ranks of the groups $\PrimSO\Sig r(P) \cong [P,\overline X^{\rm SO}_r]$ for arbitrary target manifolds $P$. First recall that the $H$-space $\overline X^{\rm SO}_r$ rationally splits as a product of Eilenberg-MacLane spaces
$$
\overline X^{\rm SO}_r \cong_\Q \prod_{j=1}^\infty K(\pi_j(X^{\rm SO}_r)\otimes \Q;j).
$$
This implies that for every $P$ we have
\begin{equation}\label{eq:cobSplit}
\begin{aligned}
[P,\overline X^{\rm SO}_r] \otimes \Q & \cong \bigoplus_{j=1}^\infty [P,K(\pi_j(\overline X^{\rm SO}_r)\otimes \Q;j)] \cong\\
& \cong \bigoplus_{j=1}^\infty H_j(P;\Z) \otimes \pi_j(\overline X^{\rm SO}_r)\otimes \Q
\end{aligned}
\end{equation}
and we only need to calculate the ranks of the homotopy groups of $\overline X^{\rm SO}_r$.

When $k$ is odd, Theorem \ref{thm:fibration} yields
$$
\begin{aligned}
\rk \pi_j(\overline X^{\rm SO}_r) &= \rk \pi_j(\Omega\Gamma T\gamma^{\rm SO}_{k+1}) - \rk \pi_j\left(\Omega\Gamma T\left( (r+2)\gamma^{\rm SO}_{k+1}\right)\right) =\\
& = \rk \pi^s_{j+1}(T\gamma^{\rm SO}_{k+1}) - \rk \pi^s_{j+1}\left(T\left( (r+2)\gamma^{\rm SO}_{k+1}\right)\right)=\\
&= \rk H_{j+1}(T\gamma^{\rm SO}_{k+1};\Z) - \rk H_{j+1}\left(T\left( (r+2)\gamma^{\rm SO}_{k+1}\right);\Z\right)=\\
&= \rk H_{j-k}(BSO(k+1);\Z) - \\
&\qquad\qquad - \rk H_{j+1-(r+2)(k+1)}(BSO(k+1);\Z);
\end{aligned}
$$
and when $k$ is even we get
$$
\begin{aligned}
\rk \pi_j(\overline X^{\rm SO}_r) &= \rk \pi_j\left(\Omega^2\Gamma T\left( (r+2)\gamma^{\rm SO}_{k+1}\right)\right) + \rk \pi_j(\Omega\Gamma T\gamma^{\rm SO}_{k+1}) =\\
&= \rk \pi^s_{j+2}\left(T\left( (r+2)\gamma^{\rm SO}_{k+1}\right)\right) + \rk \pi^s_{j+1}( T\gamma^{\rm SO}_{k+1}) =\\
&= \rk H_{j+2}\left(T\left( (r+2)\gamma^{\rm SO}_{k+1}\right);\Z\right) + \rk \pi^s_{j+1}( T\gamma^{\rm SO}_{k+1};\Z) =\\
&= \rk H_{j+2-(r+2)(k+1)}(BSO(k+1);\Z) + \\
&\qquad\qquad + \rk H_{j-k}(BSO(k+1);\Z).
\end{aligned}
$$
Substituting the ranks of the homology groups ${H_*(BSO(k+1);\Z)}$ into the formula \eqref{eq:cobSplit} we obtain the following expressions:
\begin{cor}
Denote by $p_{\leq t}(m)$ the number of partitions of $m$ into positive integers between $1$ and $t$ (in particular, $p_{\leq t}(m)=0$ whenever $m$ is not a nonnegative integer).
\begin{enumerate}[a)]
\item If $k$ is odd,
$$
\begin{aligned}
& \rk \PrimSO\Sig r(P) = \sum_{j=1}^\infty \rk H_j(P,\Z) \times \\
&\times \left( p_{\leq \frac{k-1}{2}}\left(\frac{j-k}{4}\right)+ p_{\leq \frac{k-1}{2}}\left(\frac{j-2k-1}{4}\right) - \right.\\
&\left. - p_{\leq \frac{k-1}{2}}\left(\frac{j+1-(r+2)(k+1)}{4}\right)+  p_{\leq \frac{k-1}{2}}\left(\frac{j-k-(r+2)(k+1)}{4}\right)\right).
\end{aligned}
$$
\item If $k$ is even,
$$
\begin{aligned}
\rk \PrimSO\Sig r(P) &= \sum_{j=1}^\infty \rk H_j(P,\Z) \times \\
&\!\!\!\!\!\!\!\!\!\!\!\!\!\!\times \left( p_{\leq \frac{k}{2}}\left(\frac{j-k}{4}\right) + p_{\leq \frac{k}{2}}\left(\frac{j+2-(r+2)(k+1)}{4}\right) \right).
\end{aligned}
$$
\end{enumerate}
\end{cor}
Analogous computations can be performed to obtain the ranks of the prim $\Sig r$ cobordism groups in the non-oriented and in the quaternionic cases as well.

\subsection{Geometric interpretation of Theorem \ref{thm:fibration}}

Denote by $\overline X_\infty$ the classifying space of all prim maps (of codimension $k$). Clearly $\overline X_\infty = \underset{r \to \infty}{\lim} \overline X_r$. By considering the immersion lifts of the involved maps it is easy to see that $\overline X_\infty = \Omega \Gamma T\gamma_{k+1}$.

\begin{lemma}
The homotopy exact sequence of the fibration in Theorem \ref{thm:fibration} $a)$ can be identified with the homotopy exact sequence of the pair $(\overline X_\infty,\overline X_r)$.
\end{lemma}

\begin{proof}
We first construct a map
\begin{align*}
\alpha : \pi_{n+k+1}(\overline X_\infty,\overline X_r) \to & \pi_{n+k}\left(\Omega^2\Gamma T \left((r+2)\gamma_{k+1}\right)\right) =\\
&\qquad = \pi_{n+k+1}\left(\Omega\Gamma T \left((r+2)\gamma_{k+1}\right)\right).
\end{align*}
The relative homotopy group $\pi_{n+k+1}(\overline X_\infty,\overline X_r)$ can clearly be identified with the relative cobordism group of those prim maps $F:(N^{n+1},\partial N) \to (\R^{n+k+1}_+,\R^{n+k})$ that possess the property of $F|_{\partial N} : \partial N \to \R^{n+k}$ being a $\Sigma^{1_r}$-map. The map $\alpha$ associates to the cobordism class of $F$ the cobordism class of its $\Sigma^{1_{r+1}}$-points equipped with the normal structure of its immersion lift, which is a splitting of the normal bundle into $(r+2)$ isomorphic bundles. This kind of maps is classified by the space $\Omega\Gamma T\left((r+2)\gamma_{k+1}\right)$. Thus we obtain a chain map of the homotopy exact sequence of the pair $(\overline X_\infty,\overline X_r)$ to that of the fibration of Theorem \ref{thm:fibration} $a)$. The five lemma implies that $\alpha$ is an isomorphism.
\end{proof}

Theorem \ref{thm:fibration} $b)$ and $c)$ states that this exact sequence splits rationally.
To elaborate, if $k$ is odd then the sequence
$$
0 \to \pi_{n+k}(\overline X_r)\to\pi_{n+k}(\overline X_\infty) \to \pi_{n+k}(\overline X_\infty,\overline X_r) \to 0
$$
is exact at the middle term and has finite homology everywhere else. Geometrically this means that a prim map is rationally cobordant to a $\Sigma^{1_r}$-map exactly if its $\Sigma^{1_{r+1}}$-singularity stratum is rationally null-cobordant.

If $k$ is even, then the same is true for the sequence
$$
0 \to \pi_{n+k+1}(\overline X_\infty,\overline X_r)\to\pi_{n+k}(\overline X_r) \to \pi_{n+k}(\overline X_\infty) \to 0
$$
and it geometrically means that any prim map is rationally cobordant to a $\Sigma^{1_r}$-map. Furthermore a $\Sigma^{1_r}$-map that is rationally null-cobordant as an arbitrary prim map is determined up to rational cobordism by the rational cobordism class of the $\Sigma^{1_{r+1}}$-stratum of any prim map that it bounds rationally.

\subsection{Quaternionic prim maps}\label{section:quaternion}

\bigskip

Given an $(n+3)$-dimensional manifold $P^{n+3}$ let us consider immersions of closed cooriented $n$-dimensional manifolds immersed in $P \times \R^1$ with a quaternionic structure on their normal bundle. This means that each normal fibre is provided with a quaternionic structure, and the gluing maps respect this structure. If this structure group preserves the norms of the normal vectors as well (which can always be assumed), then this structure group is the symplectic group $Sp(1) \cong Spin(3) \cong S^3$, the group of unit quaternions.

The cobordism group of such immersions is in one-to-one correspondence with the set of homotopy classes $\left[SP,\Gamma \HH P^\infty\right]$; in particular taking $P=S^{n+3}$ yields a group isomorphic to

$$\pi_{n+4}(\Gamma \HH P^\infty) = \pi^s_{n+4}(\HH P^\infty).$$

Completely analogously to the codimension $2$ oriented case (when a complex structure can be defined on the normal bundle, see \cite{sing2}) we have that if the hyperplane projection of such an immersion is a $\Sig{r}$-map (i.e. it has no singularity $\Sig{i}$ for $i > r$), then the normal bundle of the immersion can be pulled back from $\HH P^r$. The inverse is also true up to regular homotopy: if the immersion has its normal bundle pulled back from the canonical quaternionic line bundle over $\HH P^r$, then it can be deformed by a regular homotopy into an immersion such that its hyperplane projection is a $\Sig{r}$-map. We shall call such prim maps \emph{quaternionic $\Sig{r}$-prim maps}.

The cobordism group of such maps into $P$ can be defined in a standard way and will be denoted by $\PrimQ\Sig{r}(P)$.

Let $\bar X^{Sp}_r$ denote the classifying space of these cobordism groups, so that it satisfies

$$\PrimQ\Sig{r}(P) =  [\dot P,\bar X^{Sp}_r]_*$$

Here $\dot P$ denotes the one-point compactification of $P$ (if $P$ itself is compact, then this is the disjoint union of $P$ and an extra point); $[-,-]_*$ denotes the set of pointed homotopy classes.

Finally, in analogy with the complex (codimension $2$) case of \cite{sing2} we obtain that the classifying space $\bar X^{Sp}_r$ admits the representation

$$\bar X^{Sp}_r = \Omega \Gamma \HH P^{r+1}.$$

The so-called singularity spectral sequence (see \cite{sing2} for details) in homotopy groups that arises from the sequence of fibrations
$$
\bar X^{Sp}_{r-1} \subset \bar X^{Sp}_r \subset \bar X^{Sp}_{r+1} \subset \dots
$$
coincides (after a shift of the indices) with the spectral sequence in stable homotopy groups of the filtration
$$
\HH P^0 \subset \HH P^1 \subset \dots \HH P^r\subset \dots  \subset \HH P^\infty
$$

The first page of this spectral sequence is 

$$
E^1_{p,q} = \pi^s_{p+q} (\HH P^p/\HH P^{p-1}) = \pi^s_{p+q}(S^{4p}) = \pi^s(q-3p)
$$

\bigskip
\vbox{
$$
\xymatrix{
q=10 & \pi^s(7) \cong \Z_{240} & \pi^s(4)=0 & \Z_2 \\
q=9 & \pi^s(6) \cong \Z_2\langle \nu^2\rangle & \pi^s(3) \cong \Z_{24} \ar@{->>}[l] & \Z \ar[l] \ar[ull] \\
q=8 & 0 & 0 & \\
q=7 & 0 & 0 & \\
q=6 & \pi^s(3) \cong \Z_{24} & \Z \ar@{->>}[l] & \\
q=5 & \pi^s(2) \cong \Z_2 & & \\
q=4 & \pi^s(1) \cong \Z_2 & & \\
q=3 & \Z & \hspace*{60pt} & \hspace*{0pt}\\
& p=1 & p=2 & p=3 \\
}
$$
\vspace*{-128mm}
$$
\hspace*{18pt}\begin{array}{c||c|c|c|}
\cline{2-4}
\hspace*{60pt} & \hspace*{84pt} & \hspace*{84pt} & \hspace*{48pt} \\[27pt]
\cline{2-4}
 & & & \\[27pt]
\cline{2-4}
 & & & \\[27pt]
\cline{2-4}
 & & & \\[27pt]
\cline{2-4}
 & & & \\[27pt]
\cline{2-4}
 & & & \\[27pt]
\cline{2-4}
 & & & \\[27pt]
\cline{2-4}
 & & & \\[27pt]
\hhline{~||=|=|=|}
\end{array}
\hspace*{70pt}
$$
}
\bigskip

The only non-finite groups among the groups $E^1_{p,q}$ are hence those on the line $q= 3p$, these are all $\Z$.

\begin{lemma}
The group $E^\infty_{p,3p} \cong \Z$ considered as a subgroup of $E^1_{p,3p}\cong \Z $ has index equal to the order of the cokernel of the stable Hurewicz homomorphism $ \pi^s_{p+q} (\HH P^\infty) \to H_{p+q}(\HH P^\infty)$.
\end{lemma}

\begin{proof}
Consider the degenerate homological spectral sequence starting with ${\bf {H}} E^1_{p,q} = H_{p+q} (\HH P^p/\HH P^{p-1})$ (which is $\Z$ if $q=3p$ and $0$ otherwise) for the same filtration and the map from the spectral sequence $E^*_{p,q}$ into it induced by the stable Hurewicz homomorphisms on each page. On the first page we have isomorphism of the groups $E^1_{p,3p} \cong {\bf H} E^1_{p,3p} $,  both isomorphic to $\Z$. On the final page we have $E^\infty _{p,3p} = \Z$ identified with the free part of $\pi^s_{4p}(\HH P^\infty)$, and it is mapped into ${\bf {H}} E^\infty_{p,3p} = H_{4p}(\HH P^\infty) = \Z$, the image of this homomorphism being the same as the image of the stable Hurewicz homomorphism. 
\end{proof}

\begin{cor}
The product of the orders of the images of the differentials $d^r_{p,3p}: \Z = E^r_{p,3p} \to E^r_{p-r, 3p+r-1}$ 
(taken for all $r$ from $1$ to infinity) is equal to the index of the image of the stable Hurewicz homomorphism 
$\pi^s_{4p}(\HH P^\infty) \to H_{4p}(\HH P^\infty)$.
\end{cor}

Segal \cite{Segal} has determined this index:

\begin{theorem}{\cite[Theorem 1.1.]{Segal}}
The image of $\pi^s_{4p}(\HH P^\infty) \to H_{4p}(\HH P^\infty)$ is $h(p)\cdot \Z$, where
\begin{itemize}
\item $h(p) = (2p)!$ for $p$ even and
\item $h(p) = (2p)!/2$ for $p$ odd.
\end{itemize}
\end{theorem}

It follows that all the differentials of the fragment of the spectral sequence seen on the diagram are epimorphic modulo the $2$-primary torsion part, and the differential $d^1_{2,6}:E^1_{2,6} \to E^1_{1,6}$ is (truly) epimorphic. In particular we obtain that the boundaries of the normal forms of the Morin maps of type $\Sigma^{1,0}$ and $\Sigma^{1,1,0}$ in codimension $3$ give a generator of the stable homotopy group of spheres $\pi^s(3)$ and a generator of the odd torsion of the stable homotopy group of spheres $\pi^s(7)$, respectively. In the Appendix we describe the first of these classes in more detail.

We also get that the torsion parts of $\pi^s_i(\HH P^{\infty})$ in the range $i \le 11$ are $2$-primary groups, and hence so are those of the cobordism groups $\PrimQ\Sig{3}(\R^{n+3})$ for $n\le 7$.

These cobordism groups are finitely generated, and are infinite  precisely when $n \equiv 0$ mod $4$. In fact we can prove the following more general theorem that determines the rational homotopy type of the classifying space $\overline X^{Sp}_r$:

\begin{theorem}
$$\overline X^{Sp}_r \cong_\Q S^3 \times S^7 \times \dots \times S^{4r+3}.$$
\end{theorem}

\begin{proof}
We repeat the argument that determines the rational homotopy type of $\C P^m$ (which we learned from D. Crowley); we show by induction that the stable rational homotopy type of $\HH P^m$ is that of $S^4 \vee S^8 \vee \dots \vee S^{4m}$. By the induction hypothesis $\HH P^{m-1} \cong_\Q S^4 \vee S^8 \vee\dots \vee  S^{4m-4}$ and the stable homotopy class of the attaching map of the top dimensional cell of $\HH P^m$ is defined by stable maps from $S^{4m-1}$ to $S^{4i}$ for $i=1, 2, \dots, m-1$, and all these stable maps are rationally trivial (they have finite order).

Consequently $\overline X^{Sp}_r = \Omega  \Gamma \HH P^{r+1} \cong_\Q \Omega \Gamma (S^4\vee S^8 \vee \dots \vee S^{4r+4}) \cong \Gamma S^3 \times \Gamma S^7 \times \dots \times \Gamma S^{4r+3} \cong_\Q S^3 \times S^7 \times \dots \times S^{4r+3}$ (we used that $\Omega \Gamma S = \Gamma$; $\Gamma (A \vee B) = \Gamma A \times \Gamma B$; and $\Gamma S^i \cong_\Q S^i$).

\end{proof}

\begin{remark}
In \cite{sing2}[Lemma 4] it is shown that when $k=2$, we have $\overline{X_r}=\Gamma \C P^{r+1}$. This is a special case of Theorem \ref{thm:classic} as it follows from the next lemma. Recall that if $B(k+1)=BSO(2)$, then $\zeta_S=\pi_r^*\gamma_2^{\rm SO}$, where $\gamma_2^{\rm SO}$ is the universal oriented vector bundle of rank $2$ and $\pi_r: S\left((r+1)\gamma_2^{\rm SO}\right) \to BSO(2)$ is the sphere bundle of the vector bundle $\gamma_2^{\rm SO}\oplus \dots\oplus \gamma_2^{\rm SO}$ (with $(r+1)$ summands). We denote by $\gamma_1^\C$ the universal complex line bundle (over $\C P^\infty$).
\end{remark}

\begin{lemma}
The vector bundle $\zeta_S \to S\left((r+1)\gamma_2^{\rm SO}\right)$ and the complex line bundle $\gamma^\C_1|_{\C P^r} \to \C P^r$ are homotopically equivalent in the sense that there is a homotopy equivalence $f:\C P^r \to S((r+1)\gamma_2^{\rm SO})$ such that $f^*\zeta_S$ is isomorphic as an oriented rank $2$ real vector bundle to the tautological complex line bundle over $\C P^r$.
\end{lemma}

\begin{proof}
It is well-known that $\gamma^\C_1$ can be identified with $\gamma^{\rm SO}_2$, and the tautological complex line bundle over $\C P^r$ is the restriction $\gamma_1^\C|_{\C P^r}$.

Consider the space $S^\infty \times S(\C^{r+1}) \times \C$ and the natural diagonal $S^1$-action on it: for $g\in S^1$ and $(x,y,z) \in S^\infty \times S(\C^{r+1}) \times \C$ set $g(x,y,z)=(gx,gy,gz)$.\footnote{Here $S(\C^{r+1})$ is considered to be the space of unit length complex vectors in $\C^{r+1}$.} The subspace $S^\infty \times S(\C^{r+1}) \times \{0\}$ is invariant under this action; the corresponding orbit space is $S^\infty \underset{S^1}{\times} S(\C^{r+1})$. Regarding this orbit space as an $S(\C^{r+1})$-bundle over $S^\infty/S^1$ we identify it with $S\left((r+1)\gamma_1^\C\right)$. Regarding the same orbit space as an $S^\infty$-bundle over $S(\C^{r+1})/S^1 = \C P^r$ we get that it is homotopically equivalent to $\C P^r$. The obtained homotopy equivalence between $\C P^r$ and $S\left((r+1)\gamma_1^\C\right)$ takes the tautological complex line bundle to (pullback of) the tautological complex line bundle since it extends to the entire orbit space $(S^\infty \times S(\C^{r+1}) \times \C)/S^1$, which is the total space of these bundles; this finishes the proof of the lemma.
\end{proof}

Similarly $\zeta^{Sp}_S \to S\left((r+1)\gamma^{Sp}_4\right)$ is homotopically equivalent to the vector bundle $\gamma_1^\HH|_{\HH P^r} \to \HH P^r$. This combined with Theorem \ref{thm:classic} gives that $\overline X_r^{Sp}$ is $\Omega \Gamma \HH P^{r+1}$.

\section{Proofs of Theorem \ref{thm:classic} and Theorem \ref{thm:fibration}}\label{section:proofs}

\begin{proof}[Proof of Theorem \ref{thm:classic}]
Given a generic immersion $g : M^n \looparrowright P^{n+k} \times \R^1$ we first produce sections $s_1$, $s_2$, $\dots$ of the normal bundle $\nu_g$ such that $\Sigma^{1_j}(f) = \cap_{i=1}^j s_i^{-1}(0)$, where $f=pr \circ g$, the map $pr:P\times \R^1 \to P$ being the projection.

The (positive) basis vector of $\R^1$ defines a constant vector field on $P\times \R^1$ that we call (upward directed) vertical and denote by $\up$. Project $\up$ into the normal bundle $\nu_g$ (considered as the quotient bundle $T(P\times \R^1)/dg(TM)$, and denote the obtained section by $s_1$. Note that the singularity set $\Sigma(f)$ of $f$ is precisely $s^{-1}_1(0)$, the zero set of the section $s_1$.

For a generic map $f$ the set $\Sigma(f)$ is a manifold of codimension $k+1$. Denote by $\nu_2$ the normal bundle of $\Sigma(f)$ in $M$. Note that $\nu_2 \cong \nu_g|_{\Sigma(f)}$: the tangent space of the section $s_1$ at the points of $\Sigma(f)$ is the graph of a linear isomorphism $\beta_2: \nu_2 \to \nu_g|_{\Sigma(f)}$.

In order to produce the section $s_2$ we first define a section $z_2$ of $\nu_2$ by projecting $\up$ into $\nu_2$ at the points of $\Sigma(f)$ (where $\up \in dg(TM)$, hence the definition makes sense). Applying the isomorphism $\beta_2$ we get a section $s_2' = \beta_2 \circ z_2$ of $\nu_g|_{\Sigma(f)}$. The section $s_2$ is defined as an arbitrary (continuous) extension of $s_2'$ to the entire $\nu_g$. Clearly $\Sigma^{1_2}(f)$ is the zero set of $z_2$, hence $\Sigma^{1_2}(f)=s_1^{-1}(0) \cap s_2^{-1}(0)$. We continue in the same fashion, producing sections $s_3$, $\dots$, $s_{r+1}$ such that $\Sig{j}(f)=\cap_{i=1}^{j} s_i^{-1}(0)$. In particular if $f$ is a $\Sig{r}$-map, then $\cap_{i=1}^{r+1}s_i^{-1}(0)=\emptyset$.

Note that the sections $s_2$, $\dots$, $s_r$ are not unique, but each one is chosen uniquely up to a contractible choice. The difference of any two possible choices of $s_2$ is an arbitrary section of $\nu_g$ that vanishes on $s_1^{-1}(0)$. The difference of any two possible choices of $s_3$ is arbitrary section of the normal bundle $\nu_g$ that vanishes on $s_1^{-1}(0) \cap s_2^{-1}(0)$ etc. Hence the collection of these sections defines a homotopically unique section $\alpha$ of the sphere bundle $p_r^S(g): S((r+1)\nu_g) \to M$.

Let $G_{k+1}$ denote the infinite Grassmann manifold of all $k+1$-planes in $\R^\infty$ and let $\varphi: M \to G_{k+1}$ be the map that induces $\nu_g$ from the universal bundle $\gamma_{k+1}$, furthermore let $\Phi: \nu_g \to \gamma_{k+1}$ be the corresponding fiberwise isomorphism. $\Phi$ induces a map $\Phi_S^r : S((r+1)\nu_g) \to S((r+1)\nu_{k+1})$. Consider the following diagram:
\begin{equation}\tag{$*_r$}\label{eq:inductionCD}\begin{split}
\xymatrix{
\zeta_g \ar[rrr] \ar@/_/[ddd]  \ar[dr] & & & \zeta_S \ar[ddd] \ar[dl] \\
& S((r+1)\nu_g) \ar[r]^{\Phi^r_S} \ar@/_/[d]_{p_r^S(g)} & S((r+1)\gamma_{k+1}) \ar[d] & \\
& M \ar[ur]_{\ \wt{\alpha_r}=\Phi^r_S\circ\alpha} \ar[r]_\varphi \ar@/_/[u]_{\alpha}& G_{k+1} & \\
\nu_g \ar[ur] \ar[rrr]^{\Phi} \ar@/_/[uuu]_{A} & & & \gamma_{k+1} \ar[ul] \\
}
\end{split}
\end{equation}
(since $\zeta_g = p_r^S(g)^*(\nu_g)$ and $\alpha$ is a section of the bundle map $p_r^S(g)$, the map $\alpha$ can be lifted into a vector bundle map $A: \nu_g \to \zeta_g$).
Its commutativity gives us that $\wt{\alpha_r}^*\zeta_S = \nu_g$. Thus we have obtained the proof of the following lemma:
\begin{lemma}
If $g$ is a generic immersion such that $f=pr\circ g$ is a $\Sig{r}$-map, then the normal bundle $\nu_g$ can be induced by $\wt{\alpha_r}$ from $\zeta_S$, which in turn is the pullback of $\gamma_{k+1} \to G_{k+1}$ to $S((r+1)\gamma_{k+1})$ by the projection $S((r+1)\gamma_{k+1}) \to G_{k+1}$. \qed
\end{lemma}

By the remark above about the contractible choices of the sections $s_2$, $\dots$, $s_{r+1}$ we have seen that $\alpha$ is homotopically unique and therefore the map $\wt{\alpha_r}$ is also homotopically unique. Hence $\nu_g$ can be pulled back from $\zeta_S$ in a homotopically well-defined way.

Applying the Pontryagin-Thom construction to the diagram \eqref{eq:inductionCD} we construct a map
\begin{equation}\tag{$**$}\label{eq:PrimPontryagin}
\Prim{r}(P) \to [P,\Omega\Gamma T\zeta_S]
\end{equation}
This map arises as follows: the cobordism class of an immersion $g: M \looparrowright P \times \R^1$ gives a map $SP \to \Gamma T\zeta_S$. Hence the map $f=pr\circ g$ gives a map $P \to \Omega \Gamma T\zeta_S$; this map is homotopically unique and its homotopy class is the same for any representative of the cobordism class of $g$; hence also for that of $f$. Let us denote the classifying space for prim $\Sig{r}$-cobordism by $\overline X_r$. The map \eqref{eq:PrimPontryagin} is induced by a (homotopically unique) map $\theta_r:\overline X_r \to \Omega\Gamma T\zeta_S$ between the classifying spaces. In \cite{GT} we have shown that there is a fibration $\overline X_{r-1} \to \overline X_r \overset{p_r}{\to} \Gamma S^r T\left((r+1)\gamma_k\right)$ and the map $p_r$ induces the forgetful map that sends the cobordism class of a $\Sig{r}$-map to the cobordism class of the immersed top singularity stratum. Note that the base space can be rewritten as $\Gamma S^r T\left((r+1)\gamma_k\right) = \Omega\Gamma T \left( (r+1) (\gamma_k \oplus \varepsilon^1)\right)$ .
 Hence there is a long exact sequence
\begin{align*}
\dots &\to \Prim{r-1}(SP) \to \Prim{r}(SP) \to [SP,\Omega\Gamma T(r+1)\gamma_{k+1}] \to\\
&\to \Prim{r-1}(P) \to \Prim{r}(P) \to [P,\Omega\Gamma T(r+1)\gamma_{k+1}] \to 0
\end{align*}

We now try to obtain an analogous exact sequence for the functor on the right-hand side of \eqref{eq:PrimPontryagin}.
\begin{claim}\label{claim:cofibration}
There is a cofibration $$
T\zeta_S|_{S(r\gamma_{k+1})} \subset T\zeta_S \to T(r+1)(\gamma_k\oplus \varepsilon^1).
$$
\end{claim}
To see that Claim \ref{claim:cofibration} holds, we utilize a trivial lemma:
\begin{lemma}\label{lemma:cofibration}
Let $N$ be a manifold, $A\subset N$ be a submanifold with a tubular neighbourhood $V$ and normal bundle $\nu$, and let $\xi$ be a vector bundle over $N$. Denote by $\xi_A$ and $\xi_{N\setminus V}$ the restrictions of $\xi$ to $A$ and $N\setminus V$, respectively. Then there is a cofibration of Thom spaces
$$
T\xi_{N\setminus V} \to T\xi \to T(\nu \oplus \xi_A).
$$\qed
\end{lemma}
Claim \ref{claim:cofibration} follows by applying Lemma \ref{lemma:cofibration} to $N=S\left( (r+1)\gamma_{k+1} \right)$, $A=S(\gamma_{k+1})$, $\xi=\zeta_S$ with $V=S\left( (r+1)\gamma_{k+1} \right) \setminus S(r\gamma_{k+1})$ and $\nu=r(\gamma_k \oplus \varepsilon^1)$.

Applying the functor $\Omega\Gamma$ to the cofibration of Claim \ref{claim:cofibration} we obtain a fibration
$$
\Omega\Gamma T\zeta_S \xrightarrow{\Omega\Gamma T \zeta_S|_{S(r\gamma_{k+1})}} \Omega\Gamma T\left( (r+1)(\gamma_k\oplus \varepsilon^1) \right).
$$
There are also natural maps
$$\theta_{r-1}: \overline X_{r-1} \to \Omega\Gamma T \zeta_S|_{S(r\gamma_{k+1})}\text{ and}$$
$$\theta_r: \overline X_r \to \Omega\Gamma T\zeta_S$$
that correspond to the natural transformation of functors
$$\Prim{r-1}(\cdot) \to [\cdot, \Omega\Gamma T \zeta_S|_{S(r\gamma_{k+1})}]\text{ and}$$
$$\Prim r(\cdot) \to [\cdot,\Omega\Gamma T\zeta_S]$$
given by \eqref{eq:PrimPontryagin}. We hence obtain the following map of fibrations:
$$
\xymatrix{
\overline X_{r-1} \ar[r]^{\theta_{r-1}} \ar[d] & \Omega \Gamma T\zeta_S|_{S(r\gamma_{k+1})} \ar[d]\\
\overline X_r \ar[r]^{\theta_r} \ar[d] & \Omega\Gamma T\zeta_S \ar[d]\\
\Omega\Gamma T\left((r+1)(\gamma_k\oplus \varepsilon^1)\right) \ar[r]^{\cong} & \Omega\Gamma T\left((r+1)(\gamma_k\oplus\varepsilon^1)\right)\\
}
$$
We show that this diagram commutes homotopically by proving that the diagram of the induced natural maps between the corresponding functors commutes. The commutativity of the top square is obvious. In the bottom square, the normal structure encoded in $(r+1)(\gamma_k\oplus\varepsilon^1)$ on the left is that of the splitting of the normal bundle of $\Sigma^{1_r,0}$ into $r+1$ bundles canonically isomorphic to the restriction of the normal bundle of the immersion lift, while the normal structure on the right encodes the common zero sets of the $r+1$ sections $s_j$. While the structure on the right may be twisted with respect to that on the left by isomorphisms provided by the sections themselves, one can still recover one normal structure from the other uniquely up to homotopy.
\end{proof} 

\begin{proof}[Proof of Theorem \ref{thm:fibration}]
{\it a)} Consider the vector bundle $\gamma_{k+1}$ pulled back to the disc bundle $D\left((r+1)\gamma_{k+1}\right)$ by the projection. Note that the Thom space of this bundle is homotopy equivalent to $T\gamma_{k+1}$, while the total space of its disc bundle is that of the disc bundle $D\left((r+2)\gamma_{k+1}\right)$. Recall that after identifying $\gamma_{k+1}$ with this pullback to the disc bundle the vector bundle $\zeta_S$ is the restriction $\gamma_{k+1}|_{S((r+1)\gamma_{k+1})}$, and notice that $T\zeta_S \to T\gamma_{k+1} \to T\left((r+2)\gamma_{k+1}\right)$ is a cofibration by Lemma \ref{lemma:cofibration}. By applying the functor $\Omega\Gamma$ to it, we obtain the fibration
$$
\Omega\Gamma T\zeta_S \to \Omega\Gamma T\gamma_{k+1} \to \Omega\Gamma T\left((r+2)\gamma_{k+1}\right).
$$
It is well-known (see e.g. \cite{MosherTangora}) that when one turns the inclusion $\Omega\Gamma T\zeta_S \to \Omega\Gamma T\gamma_{k+1}$ of the fiber into a fibration, its fiber will be the loop space of the base, $\Omega^2\Gamma T\left((r+2)\gamma_{k+1}\right)$. This fibration $\Omega\Gamma T\zeta_S = \overline{X_r} \xrightarrow{\Omega^2\Gamma T\left((r+2)\gamma_{k+1}\right)} \Omega\Gamma T\gamma_{k+1}$ is the one stated by Theorem \ref{thm:fibration} $a)$.
\par
{\it b)} When $k$ is odd and we are in the oriented setting, the vector bundle $\gamma_{k+1}^{SO}$ has a nonvanishing Euler class and so does the vector bundle $(r+1)\gamma_{k+1}^{SO}$ as well, hence (using the Gysin sequence) we obtain that the projection $pr: S((r+1)\gamma^{SO}_{k+1}) \to BSO(k+1)$ induces an epimorphism in cohomology with rational coefficients (here we use the fact that $H^*(BSO(k+1))$ has no zero divisors as a subring of the polynomial ring $H^*(BT)$ with $T$ the maximal torus of $BSO(k+1)$). Consequently so does the induced map $T\zeta_S \to T\gamma^{SO}_{k+1}$ of Thom spaces, therefore the cofibration  $T\zeta_S \to T\gamma^{SO}_{k+1} \to T((r+2)\gamma^{SO}_{k+1})$ splits homologically and so the long exact sequence of the fibration $\Gamma T\zeta_S \to \Gamma T\gamma^{SO}_{k+1} \to \Gamma T((r+2)\gamma^{SO}_{k+1})$ in homotopy -- which coincides with the long exact sequence of the original cofibration in stable homotopy -- splits rationally as well. Since all the involved spaces are H-spaces and thus rationally products of Eilenberg-MacLane spaces, this implies that $\Omega\Gamma T\gamma^{SO}_{k+1} \cong_\Q \Omega\Gamma T\zeta_S \times  \Omega\Gamma T\left((r+2)\gamma^{SO}_{k+1}\right) =\overline X^{SO}_r \times  \Gamma T\left((r+2)\gamma^{SO}_{k+1}\right)$ as claimed.
\par
{\it c)} When $k$ is even (and we are still in the oriented setting), the vector bundle $\gamma_{k+1}^{SO}$ has vanishing Euler class and so does the vector bundle $(r+1)\gamma_{k+1}^{SO}$ as well, hence the projection $pr: S((r+1)\gamma^{SO}_{k+1}) \to BSO(k)$ induces a monomorphism in cohomology with rational coefficients. Consequently so does the induced map $T\zeta_S \to T\gamma^{SO}_{k+1}$ of Thom spaces. Extending the cofibration $T\zeta_S \to T\gamma^{SO}_{k+1} \to T((r+2)\gamma^{SO}_{k+1})$ to form the Puppe sequence
\begin{align*}
T\zeta_S \to T\gamma^{SO}_{k+1} &\to T((r+2)\gamma^{SO}_{k+1}) \to\\
&\to ST\zeta_S \to ST\gamma^{SO}_{k+1} \to ST((r+2)\gamma^{SO}_{k+1}) \to \dots
\end{align*}
we observe that the induced map $H^*\left(ST\left((r+2)\gamma^{SO}_{k+1}\right)\right) \to H^*\left( ST\gamma^{SO}_{k+1} \right)$ is also zero, therefore the cofibration  $T\left((r+2)\gamma^{SO}_{k+1}\right) \to ST\zeta_S \to ST\gamma^{SO}_{k+1}$ splits homologically and so the long exact sequence of the fibration
$$\Omega^2\Gamma T\left((r+2)\gamma^{SO}_{k+1}\right) \to \Omega^2\Gamma ST\zeta_S \to \Omega^2\Gamma ST\gamma^{SO}_{k+1}$$
in homotopy splits rationally as well. But since $\Omega^2\Gamma ST\zeta_S = \Omega\Gamma T\zeta_S$ and $\Omega^2\Gamma ST\gamma^{SO}_{k+1} = \Omega\Gamma T\gamma_{k+1}$ are H-spaces, we have
$$\overline X_r =\Omega\Gamma T\zeta_S \cong_\Q \Omega \Gamma T\gamma_{k+1} \times \Omega^2\Gamma T\left((r+2)\gamma_{k+1}\right)$$
as claimed.
\end{proof}

\section{Appendix: The generator of $\pi^s(3)$ via singularity theory}

Consider the normal form of a Whitney umbrella map $U:\R^4 \to \R^7$ given by the coordinate functions
\begin{align*}
(x,t_1,t_2,t_3) &\mapsto (y_1,y_2, y_3,z_1,z_2,z_3,z_4)\\
y_m &= t_m \qquad m=1,2,3\\
z_m &= t_mx \qquad m=1,2,3\\
z_4 &= x^2
\end{align*}
By adding the eighth coordinate function $z_5=x$ we lift this map to an embedding $\tilde U:\R^4 \to \R^8$. Consider the restriction of $\tilde U$ to $\tilde D^4 = U^{-1}(D^7)$ and denote by $\tilde S^3$ the boundary of $\tilde D^4$. Note that $\tilde D^4$ is diffeomorphic to the standard ball $D^4$. The map $\tilde U$ is an embedding, hence $\tilde U(\tilde D^4)$ is an embedded ball in $\R^8$. Its normal bundle admits a homotopically unique quaternionic line bundle structure. The direction of the added $8$th coordinate line $(z_5)$ is not tangent to $\tilde D^4$ at the points of $\tilde S^3$ (actually at any point of $\tilde D^4$ except for the origin). Hence the classifying map sends this subset to a ball neighbourhood of $\HH P^0 \subset \HH P^1$ and induces the normal bundle of $\tilde U$ from $\gamma^H_1$ in the following manner. Let $v$ denote the image of the upward-pointing vector $\partial/\partial z_5$ in the normal bundle of $\tilde D^4$ under the natural projection $T\R^8 \to T\R^8/d\tilde U(T\tilde D^4) = \nu_{\tilde U}$ . The classifying map sends $v$ to a fixed section $s$ of $\gamma^H_1$ (that does not vanish on the given neighbourhood) and extends this bundle map to respect the quaternionic structure (see the proof of the second Claim in \cite[Section 3]{sing2}). In particular, the trivialization $(s,is,js,ks)$ of $\gamma^H_1$ over $\HH P^0$ that gives the generator of $\pi^s(0) = \Z$ is pulled back to the trivialization $(v,iv,jv,kv)$ of the normal bundle of $\tilde U$, and the element in $\pi^s(3) \approx \Z_{24}$ represented by this framed manifold $\tilde S^3$ is the image of the differential $d^1_{2,6}: E^1_{2,6} \cong \Z \to E^1_{1,6} = \pi^s(3) \cong \Z_{24}$ in the spectral sequence of subsection \ref{section:quaternion} evaluated on a generator of $E^1_{2,6} \cong \Z$. Since the differential $d^1_{2,6}$ is surjective, it sends the generator of $E^1_{2,6}$ to a generator of $E^1_{1,6} = \pi^s(3) \cong \Z_{24}$.

\end{document}